\documentclass[letterpaper, 10 pt, conference]{ieeeconf}

\usepackage{hyperref}
\usepackage{pifont}
\newcommand{\tick}{\ding{51}}%
\newcommand{\cross}{\ding{55}}%

\usepackage{amsmath}
\usepackage{amssymb}
\usepackage{graphicx}
\usepackage{color}
\usepackage{xcolor}
\usepackage{subfigure}
\usepackage{tikz}
\usetikzlibrary{automata,shapes}

\newcommand{\R}{\mathbb{R}}
\newcommand{\N}{\mathbb{N}}
\newtheorem{theorem}{Theorem}
\newtheorem{lemma}{Lemma}
\newtheorem{corollary}{Corollary}
\newtheorem{definition}{Definition}
\newtheorem{example}{Example}
\newtheorem{problem}{Problem}
\newtheorem{remark}{Remark}

\usepackage{cleveref}

\usepackage{sriram-macros}

\IEEEoverridecommandlockouts                              % This command is only
\title{ \bf
Stabilization of polynomial dynamical systems\\ using linear programming based on Bernstein polynomials}
\author{Mohamed Amin Ben~Sassi$^{1}$ and Sriram Sankaranarayanan$^{1}$
\thanks{$^{1}$M.A.Ben Sassi and  S. Sankaranarayanan are in the Department of Computer Science, University of Colorado, Boulder, CO, USA}
  }          
\begin{document}
\maketitle
\thispagestyle{empty}
\pagestyle{empty}

\begin{abstract}
In this paper, we deal with the problem of synthesizing static output
feedback controllers for stabilizing polynomial systems.  Our approach
jointly synthesizes a Lyapunov function and a static output feedback
controller that stabilizes the system over a given subset of the
state-space. Specifically, our approach is simultaneously targeted
towards two goals: (a) asymptotic Lyapunov stability of the system,
and (b) invariance of a box containing the equilibrium. Our approach
uses Bernstein polynomials to build a linear relaxation of polynomial
optimization problems, and the use of a so-called ``policy iteration''
approach to deal with bilinear optimization problems.  Our approach
can be naturally extended to synthesizing hybrid feedback control laws
through a combination of state-space decomposition and Bernstein
polynomials.  We demonstrate the effectiveness of our approach on a
series of numerical benchmark examples.

%
%{Bernstein Polynomials, Box invariant , Lyapunov Functions, Stability, Controller synthesis}
\end{abstract}
%%%%%%%%%%%%%%%%%%%%%%%%%%%%%%%

%%%%%%%%%%%%%%section A%%%%%%%%%
\section{Introduction}
The problem of designing stabilizing controllers for nonlinear
dynamical systems is of great importance. In this paper, we study the
problem of synthesizing static output feedback controllers for
polynomial systems by solving a polynomial optimization problem to
directly obtain the controller along with the associated Lyapunov
functions that yields the proof of stability. 

Our approach inputs the description of a polynomial system and a
desired region $R$ to be stabilized. It then proceeds to find a static
output feedback control law and an associated Lyapunov function to
ensure local stability in $R$. Simultaneously, we ensure that the
region $R$ is an invariant of the resulting closed loop system. Our
approach assumes a given structure for the feedback as a polynomial
function of the outputs of the system.  Furthermore, we assume a
polynomial template form for the unknown Lyapunov function. We proceed
to encode the conditions for the Lyapunov function, obtaining a hard
polynomial optimization problem that involves the coefficients of the
Lyapunov functions and those of the feedback. 

The second part of the paper iteratively solves this optimization
problem through an iterative method variously called ``V-K''
iteration~\cite{Ghaoui1994} or policy iteration~\cite{Gaubert2007}.
The $i^{th}$ iteration of the approach selects a positive definite
polynomial $V_i$ and a feedback law $u_i$. Ideally, we require
${V_i'}$ to be negative definite inside the region $R$ for $V_i$ to
be a Lyapunov function guaranteeing asymptotic stability. Failing
this, we first search for a new positive definite polynomial $V_{i+1}$
whose Lie derivative ${V_{i+1}'}$ has a larger maxima inside $R$
\emph{fixing} $u_i$, and adjust to a new feedback law $u_{i+1}$ that
improves the maximal value of $V_{i+1}'$ inside $R$. Each
iteration is reduced to solving a Linear Programming (LP) problem
using Bernstein polynomials combined with a reformulation
linearization technique~\cite{Bensassi+Sriram}. It is well-known that
policy iteration  does not necessarily converge to a global minimum, in
general.  However, our evaluation over a wide variety of benchmark
examples shows that our approach is effective at converging to a
global minimum by discovering an appropriate feedback law $u^*$ and an
associated Lyapunov function $V^*$.

Automatic static output feedback design, or more generally, finding
feedback that satisfies given structural constraints is well-known to
be a hard problem in general.  In fact, static output feedback
stabilization of linear systems yields \emph{bilinear matrix
inequalities} (BMIs) rather than LMIs. A direct approach given by
Henrion et al.~\cite{Henrion2005} uses the characteristic polynomial of
the transfer function matrix, and derives constraints that ensure the
Hermite stability criterion for this matrix.  As a result, they obtain
a system of PMI (polynomial matrix inequalities), that is solved using a local
optimization solver (PENBMI). In contrast, an indirect approach
reduces the non convex BMIs to a series of convex LMIs. This was
proposed as the so-called $V-K$ iteration was proposed by El Ghaoui
and Balakrishnan~\cite{Ghaoui1994}. The approach iteratively solves a
bilinear problem by fixing one set of variables while modifying the
other to result in a decrease in the objective values. The iteration
alternates between the two sets of variables, until reaching a
feasible solution. Our goal is to use this technique for polynomial
systems while replacing BMI and LMI with linear and bilinear programs
that can be solved more efficiently. A similar idea for solving
bilinear problems appears in the work of Gaubert et
al.~\cite{Gaubert2007}, for finding invariants for discrete-time
systems. Therein, the idea is called \emph{policy iteration}.  In this
work, we will call our approach \emph{policy iteration}, as well.  The
main differences between our work and that of El Gahoui et al. lie in
our focus on polynomial systems, yielding more general polynomial
optimization problems that involve the ``V'' variables relating to the
Lyapunov function and the ``K'' variables relating to the feedback.
Yet, by using policy iteration, we can separately focus on problems
with a single set of variables at a time and use linear programming
relaxations through a combination of Bernstein polynomials and
reformulation linearization, discussed in our earlier
work~\cite{Bensassi+Sriram}.

 Existing approaches to stabilizing polynomial systems rely on
linearization around the equilibrium. However, linearization can
sometimes fail to be controllable, or yield region of stability that
is much smaller than desired. Furthermore, the output feedback
stabilization for a linear system (or finding a feedback law
satisfying a given structure) yields non-convex problems that are no
easier to solve.  Another class of methods (more related to our work) consists on
reducing the problem to a set of LMIs or Sum-Of-Squares (SOS)
formulations (see~\cite{Zhao2009,Nguang2011} and references
therein). In~\cite{Nguang2011}, an iterative SOS approach is
proposed. This approach uses the Schur complement to produce a set of
BMIs relaxed to an SOS problem. More precisely, an additional design
nonlinear term $\epsilon(x)$ is introduced, and causes bilinearity. An
iterative approach is then obtained by fixing a guess for
$\epsilon(x)$ and iteratively updating it until feasibility is
obtained. Once again, the major problem arises from the fact that the
Lyapunov function and a static output feedback are needed
simultaneously. Other approaches to controlling polynomial systems
include the use of nonlinear optimal control techniques, feedback
linearization, backstepping, and exact linearization. However, these
techniques rely on the system being of a certain form and mostly
involve state-feedback.  A detailed comparison of the relative
advantages of the direct approach presented here with other approaches
to nonlinear stabilization will form an important part of our future
work. 
\section{Problems formulation and polynomial optimization problems}\label{sec:problem}
\subsection{Problem formulation}
%\section{Problem Formulation}

\begin{figure}[t]
\begin{center}
\begin{tikzpicture}
\matrix[every node/.style={rectangle, draw=black}, pnt/.style={minimum size=0pt, inner sep=0pt,circle,draw=white}, row sep=15pt, column sep=20pt]{
  \node[pnt](p0){}; & & \node(n0){$\dot x = f(x)+g(x) u$}; &  & \node[pnt](p1){}; \\
 \node[pnt](p3){};  &  & \node[fill=blue!20](n1){$u=H(y)\theta$}; &  & \node[pnt](p2){};\\
};
\path[->] (n0) edge node[above]{$y = h(x)$} (p1)
(p1) edge[-] (p2)
(p2) edge (n1)
(n1) edge (p3)
(p3) edge[-] (p0)
(p0) edge node[above]{$u$} (n0);

\end{tikzpicture}

\end{center}
\caption{ Overall structure of the controller synthesis problem considered.}\label{Fig:control-structure-dia}
\end{figure}
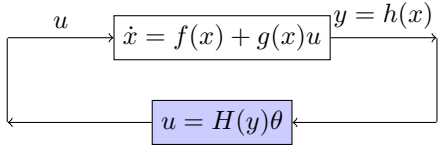
In this work, we consider a  nonlinear control-affine system subject to input constraints :
\begin{equation}
\label{eq:ode}
\left\{
\begin{array}{ll}
\dot x(t) = f(x(t))+g(x(t)) u(y(t)),\; u\in \mathcal U. \\
y(t)=h(x(t)).
\end{array}\right.
\end{equation}
wherein $x\in \reals^n$ represents the state variables, $u\in \mathcal
U$ represents the control inputs ranging over a compact set $U
\subseteq \R^p$, and $y\in \R^q$ are the outputs.

We assume that the functions $f:\R^n \rightarrow \R^n$, $h:\reals^n
\rightarrow \R^q$ and the control matrix $g:\R^n \times
\R^m\rightarrow \R^{(n \times p)}$ defining the dynamics of the system
are multivariate polynomial maps.  The set of inputs $\mathcal U$ is a
convex compact polytope: $\mathcal U=\left\{u \in \R^p |\;
  \alpha_{\mathcal U,k}\cdot u \le \beta_{\mathcal U,k} ,\; \forall
  k\in \mathcal K_{\mathcal U} \right\}$ where $\alpha_{\mathcal
  U,k}\in \R^p$, $\beta_{\mathcal U,k} \in \R$ and $\mathcal
K_{\mathcal U}$ is a finite set of indices.  Finally, we assume that
$x^*=0_n$ is an equilibrium for the system~\eqref{eq:ode}, i.e
$f(0_n)+g(0_n)u(h(0_n))=0_n$.

We define a region of interest  $R$ as a hyper-rectangle, $R:\
[\underline{x_1},\overline{x_1}] \times \dots \times
[\underline{x_n},\overline{x_n}]$ with
$\underline{x_k}<\overline{x_k}$ for all $k\in \{1,\dots,n\}$.

\paragraph{Stabilizing Feedback:} In this work, we assume that the
desired feedback is given by a function $u:  \R^q \rightarrow \scr{U}$ mapping
outputs $y$ to control inputs $u$ to yield a  closed-loop system 
\begin{equation}
\label{eq:ode1}
\dot x = f(x)+g(x) u(y),\ \ 
y = h(x)
\end{equation}

We require that the closed loop system~\eqref{eq:ode1} be
asymptotically stable in $R$.  This is achieved by ensuring two
important properties.
\begin{problem}[Existence of Local Lyapunov Function]\label{pb1}
  The system~\eqref{eq:ode1} has a local Lyapunov function $V(x)$ in
  the region $R$ such that
\begin{enumerate}
\item $V(x)$ is \emph{positive definite} over $R$, i.e, $V(x) > 0 $ for
  all $x \in R \setminus \{ 0_n \}$ and $V(0_n) = 0$.
\item $\frac{dV}{dt} =  \nabla V \cdot (f(x) + g(x) u(h(x)) )$ is negative definite over
$R$.
\end{enumerate}
\end{problem}
As such, a local Lyapunov function inside $R$ guarantees that the
system~\eqref{eq:ode1} is asymptotically stable in some neighborhood
$N$ of $0_n$, where $N \subseteq R$. Specifically, $N$ contains the largest
sublevel set of $V$ inside $R$ as the stability region, but does not have
to include $R$.  To ensure
that the system is stable inside all of $R$, we additionally require positive
invariance of $R$.
\begin{problem}[Positive Invariance of $R$]\label{pb2}
The system~\eqref{eq:ode1} is $R$-\emph{invariant}, iff all trajectories with $x(0)\in R$
satisfy $x(t)\in R$ for all $t\ge 0$.  
\end{problem}
Finding a feedback $u(y)$ that solves problems~\ref{pb1} and
~\ref{pb2} ensures asymptotic stability in the whole region $R$.

\paragraph{Feedback Structure}
Finally, we consider feedback functions that conform to a given fixed
structure. In other words, we consider feedback functions of the
following form
$$
u(y)=H(y)\cdot \theta=\scr{H}(x)\cdot \theta
$$
where $\theta \in \R^q$ is a set of \emph{gain parameters} to be
determined by the synthesis procedure, the matrix $H:\R^n \rightarrow
\R^{(p \times q)}$ is a given multivariate polynomial map that
specifies the controller structure. Often, $H$ is specified to
include all monomial terms up to a given degree. However, more complex
situations such as \emph{decentralized control} may involve choosing
specific structure for $H$. Figure~\ref{Fig:control-structure-dia}
depicts the structure of the controller schematically.

Let $\scr{H}(x): H(h(x))$ be the equivalent map as a function of the
state variables.  The input constraints (i.e. for all $x\in R$, $u
\in \mathcal{U}$) is then equivalent to
\begin{equation}
\label{eq:inputcons}
 \forall k\in \mathcal K_U,\; \forall x\in R,\; \alpha_{U,k}\cdot \mathcal{H}(x)\theta \le \beta_{U,k}.
\end{equation}
Let $O$ represent the values of $\theta$ that satisfy Eq.~\eqref{eq:inputcons}.
Under these assumptions, the dynamics of the controlled system~(\ref{eq:ode1}) can be rewritten under the form
$$
\dot x(t) = f(x(t))+G(x(t)) \theta ,
$$
where the matrix of polynomials $G(x)=g(x)\mathcal{H}(x)$, and $\theta \in O$.

\subsection{Reduction to polynomial optimization problems (POP)}\label{Sec:lp-Bernstein-relaxations}

The first step is to fix a template form  for the Lyapunov function $V$. We 
assume a polynomial form:
$$V=V_c(x)=\displaystyle{ \sum_{|\alpha| \le D} c_{\alpha} x^{\alpha}} \,,$$ 
where $\alpha \in \N^n$, $|\alpha|=\sum_{i} \alpha_i$,
$c:\ (c_{\alpha})_{|\alpha| \le D}$ are the unknown coefficients of the
Lyapunov function and $D \in \N$ is the maximal degree.

We now focus on solving Problem~\ref{pb1}.
For a relatively small $\epsilon > 0$, this problem can be formulated as follows:
\begin{enumerate}
\item Find a feasible set $C$ s.t 
\[ C:\{ c\ |\ \displaystyle{\min_{x\in R} V_c(x)-\epsilon ||x||^2 \ge 0 }\}\]
\item Find feasible sets $c \in C'$ and $\theta \in O'$ s.t forall $c \in C'$ and
$\theta \in O'$,
\[ \displaystyle{\min_{x\in R} -\nabla V_c \cdot (f+G\theta)-\epsilon ||x||^2 \ge 0} \]

\end{enumerate}
Recall the set $O$ from Eq.~\eqref{eq:inputcons}.
\begin{theorem}
If $C\bigcap C' \neq \emptyset $ and $ O\bigcap O' \neq \emptyset $, then each $c^*\in C\bigcap C' $ and $\theta^*\in O\bigcap O' $
solves the local Lyapunov function existence problem (Problem~\ref{pb1}).
\end{theorem}
\begin{proof}
It is easy to see that the first condition will imply that $V_c$ will be positive definite, the second one implies that its derivatives $\frac{dV}{dt}$
is negative definite. The last condition implies that the controller is admissible i.e $u\in \mathcal U$.
\end{proof}
%For the invariance problem, we will need the following notations~\cite{Belta06} concerning facets of a given rectangle $R_n=\prod_{k=1}^{k=n}[a_k,b_k]$ and its corresponding outer normals:
%\begin{itemize}
%\item $\xi_k:\{a_k,b_k\}\mapsto\{0,1\} $ when for all $k\in\{1,\ldots,n\}$, $\xi_k(a_k)=0$ and $\xi_k(b_k)=1$. 
%\item$F_{j,\xi_j(w_j)}=\{x\in R_n\; | \;x_j=w_j\}$: the set of facets of $R_n$
% where for all $j\in\{1,\ldots,n\}$, $w_j\in\{a_j,b_j\}$. 
%\item$n_{j,\xi_j(w_j)}=(-1)^{(\xi_j(w_j)+1)}e_j$: the outer normal of the facet $F_{j,\xi_j(w_j)}$ where 
%the vectors $e_j$ form the canonical basis of $R^n$.
%\end{itemize}
%In our case $R_n=R=[-1,1]^n$. 
To solve the invariance problem (Problem~\ref{pb2}), we should find a controller (i.e a coefficient vector $\theta$) ensuring that all the facets of the rectangle $R$ are blocked.
\begin{definition}[Blocked Facets]
A facet $F$ of the hyper-rectangle $R$ is said to be \emph{blocked} for the system~(\ref{eq:ode1}) if and only if 
\[  \forall x \in F ,\; n_F.(f(x)+G(x)\theta) < 0\,,\]
 where $n_F$ is its outer normal of the facet $F$.
\end{definition}
Let $\mathcal F$ denote the set of facets of the rectangle $R$, then solving Problem~\ref{pb2} can be formulated as follows :
\begin{itemize}
\item Find feasible set $O_F$ such that for all $\theta \in O_F$ s.t 
\[ \displaystyle{\min_{x\in  F} n_{F}.(f(x)+G(x)\theta)} < 0 \,,\]
 for all facet $F \in \mathcal F$.
\end{itemize}
Recall that $O$ represents the feasible set from~\eqref{eq:inputcons}.
\begin{theorem}
If $O_{\mathcal F}\bigcap O \neq \emptyset $, then each $\theta^*\in O_{\mathcal F}\bigcap O $
ensure the invariance of the rectangle $R$ and solve Problem~\ref{pb2}, where $O_{\mathcal F}=\displaystyle{\bigcap_{F\in \mathcal F}O_F}$.
\end{theorem}
\begin{proof}
Since  $\theta^*\in O_{\mathcal F}$ then all the facets of $R$ are blocked implying its invariance. The fact
that  $\theta^*\in O'$ proves that the controller is admissible.
\end{proof}

%%%%%%%%%%%%%%%%%%%%%%%%%%%%%
%\subsection{Comparison with other relaxations}\label{Sec:comparisons-relaxations}
%\input{section31.tex}
\section{Reduction to Linear and bilinear feasibility problems}
In this section, we are going to relax the previous polynomial
optimization problems to a set of linear and bilinear feasibility
problems.  For doing so, we will briefly recall a relevant result
showing how a general POP can be realxed to a linear program using
Bernstein polynomials~\cite{Bensassi+Sriram}, then we will use this
relaxation to build our linear and bilinear feasibility problems in
order to solve our two given problems.
\subsection{Linear relaxation of a POP using Bernstein polynomials}
In this section, we are going to use Bernstein polynomials to
establish lower bounds for our polynomial optimization problems (POP).
More precisely, we seek tight lower bound for the optimal solution of
the following POP:
\begin{equation}\label{eq:POP}
\text{minimize}\ p(x)\ \text{s.t.}\  x \in R\,.
\end{equation}
where $p$ is multi-variate polynomial of degree $\delta:\ (\delta_1,\dots,\delta_n)$.\\
We build a linear relaxation for problem (\ref{eq:POP}), as follows:
\begin{enumerate}
\item Change of variable $q_U$ mapping $R$ to the unit box $U=[0,1]^n$. Let $p_U=p\circ q_U$.
\item Write $p_U$ in the Bernstein basis.
\item Write an equivalent POP in the Bernstein basis.
\item Exploit properties of Bernstein polynomials to formulate
a linear programming problem whose optimum is guaranteed to lower
bound the POP in Eq.~\eqref{eq:POP}. 
\end{enumerate}
We now explain the procedure in further detail. First of all, the
mapping $q_U$ from any rectangle $R$ to the unit box $[0,1]^n$ is an
affine transformation. Therefore, the multi-variate polynomial $p_U$ is also of
degree $\delta$ and we can write:
$$p_U(y)=\displaystyle{ \sum_{\alpha \le \delta} p_{\alpha} y^{\alpha}} \text{ for all } y\in U,$$
where $(p_{\alpha})_{\alpha \le \delta}$ denotes the new coefficients
of $p_U$ in the standard monomial basis, and the order relation $\alpha
\le \delta$ is such that $\alpha_i \le \delta_i$ for all $i\in
\{1,\dots,n\}$.  By writing $p_U$ in the Bernstein basis we obtain
the following form:
$$
p_U(y)=\displaystyle{\sum_{I\le\delta }  b_{I,\delta} B_{I,\delta}(y) },
$$         
where Bernstein coefficients $(b_{I,\delta})_{I\le\delta}$ are given as follows:
\begin{equation}\label{eq:bernstein-coeff-formula}
b_{I,\delta}=\sum_{J\le I} \frac{\left(
\begin{array}{c}
 i_1 \\ j_1
\end{array}
\right) \dots
\left(
\begin{array}{c}
 i_n \\ j_n
\end{array}
\right) 
}
{\left(
\begin{array}{c}
 \delta_1 \\ j_1
\end{array}
\right) \dots
\left(
\begin{array}{c}
 \delta_n \\ j_n
\end{array}
\right) } p_J=\sum_{J\le I}\frac{\left(
\begin{array}{c}
 I \\ J
\end{array}
\right)}{\left(
\begin{array}{c}
 \delta \\ J
\end{array}
\right)}p_J.
\end{equation}
and Bernstein polynomials are as follows: 
\begin{equation}
 B_{I,\delta}(y) =\left(
\begin{array}{c}
 \delta \\ I
\end{array}
\right) 
y^{I} (1_n-y)^{\delta-I}.
\end{equation}
where $y^I=({y_1}^{i_1},\dots,{y_n}^{i_n} )$, $\delta-I=(\delta_1-i_1,\dots, \delta_n-i_n)$ and $\left(
\begin{array}{c}
 \delta \\ I
\end{array}
\right)=\left(
\begin{array}{c}
 \delta_1 \\ i_1
\end{array}
\right) \dots
\left(
\begin{array}{c}
 \delta_n \\ i_n
\end{array}
\right).$\\
For the third step it is sufficient to replace the canonic form by the Bernstein form in the optimization problem, we then get
the following optimization problem:
\begin{equation}\label{eq:POP1}
\begin{array}{ll}
\text{minimize} & \displaystyle{\sum_{I\le\delta }  b_{I,\delta} B_{I,\delta}(y) }\\
\text{s.t} & y \in U.\\
& z_I=B_{I,\delta}(y).
\end{array}
\end{equation}
The final step is now to remove the nonlinearities caused by the Bernstein polynomials
by replacing each Bernstein polynomial $B_{I,\delta}$ by a fresh variable $z_I$. In 
effect, we drop the relation $z_I = B_{I,\delta}$. To recover precision, we add
some of the known linear relations between Bernstein polynomials:
\begin{itemize}
\item Unit partition: $\displaystyle{\sum_{I\le \delta} B_{I,\delta}(y) }=1.$
\item Bounded polynomials: $0 \le B_{I,\delta}(y) \le B_{I,\delta}(\frac{I}{\delta}), \text{ for all } I\le \delta.$
\end{itemize}
By injecting these properties in (\ref{eq:POP1}), we obtain the following linear relaxation:
 \begin{equation}
\label{eq:lpbern2}
\begin{array}{rllr}
\text{minimize} & \displaystyle{\sum_{I\le \delta} b_{I,\delta} z_{I,\delta} }\\
\text{s.t} & z_{I,\delta}\in \mathbb{R},\; & I\le \delta, \\
&  0 \le z_{I,\delta}\le B_{I,\delta}(\frac{I}{\delta}),  & I\le \delta, \\
& \displaystyle{\sum_{I\le \delta} z_{I,\delta} =1},
\end{array}
\end{equation}
\begin{lemma}
The optimal value of (\ref{eq:lpbern2}) gives a lower bound for the POP (\ref{eq:POP}).
\end{lemma}

\subsection{Linear and bilinear feasibility programs for existence of Lyapunov function
(Problem~\ref{pb1})}
Let $V(x,c)$ be the assumed polynomial form for the Lyapunov
function with unknowns $c$. We first focus on 
encoding the positive definiteness of $V$ inside $R$.
We recall the sets $C,C', O,O'$ from section~\ref{Sec:lp-Bernstein-relaxations}.

First, we consider the set 
\[C: \left\{ c \left| \min_{x \in R}\left(\epsilon ||x||^2 - V(x,c) \right) \leq 0 \right.\right\}\,.\] 
 
Let $m(x)$ represent a vector of monomials involved in $\epsilon
||x||^2 - V(x,c)$ so that we may write $\epsilon ||x||^2 - V(x,c):\ \tilde{c}^t L m$, where $\tilde{c}=\left(\begin{array}{c} 1 \\ c \end{array}\right)$ for
a suitable matrix $L$.
Writing $m$ in the Bernstein
basis, we obtain $m: \mathcal{B}z$ where $z$ represents a vector of
polynomials in the Bernstein basis and $\mathcal{B}$ is a linear
transformation. Therefore, the problem~\eqref{eq:lpbern2} is written
equivalently as 
\begin{equation}\label{eq:mpt-relaxation}
\begin{aligned}
\min\  & -\tilde{c}^t\  L \scr{B}\ z \\
\mbox{s.t.}\  & A z\ \leq  b \\
\end{aligned}
\end{equation}
Let $\hat{C}$ be the set of all values of $c$ such that problem~\eqref{eq:mpt-relaxation}
with $c \in \hat{C}$ yields a non-positive optimal value. In other words, 
\begin{equation}\label{eq:rlt-parameterized}
\hat{C}:\ \left\{  c\ | \ (\forall\ z)\ A z \leq b\ \Rightarrow\ - \tilde{c}^t\ L \scr{B}\ z \ \leq 0 \right\}\,.
\end{equation}

\begin{lemma}
$\hat{C} \subseteq C$
\end{lemma}
To represent the set $\hat{C}$, we use Farkas lemma, a well known
result in linear programming, to dualize~\cref{eq:rlt-parameterized}
and obtain our first linear feasibility problem for computing $\hat{C} \subseteq C$.

\begin{lemma}\label{lem:feasibility}
 The vector $c$ is a solution to the problem in~\cref{eq:rlt-parameterized} if and only if there exist multipliers $c$, ${\lambda}$
such that 
\begin{equation}
\label{eq:feasability1}
 A^t {\lambda} = - \scr{B}^t L^t \tilde{c},\; b^t\lambda \leq 0,\; \mbox{ and }\lambda \geq 0 
 \end{equation}
\end{lemma}

Next, we consider the set $O$ encoding the input
constraints in ~\eqref{eq:inputcons}.
Let $\mathcal{H}_{i}$ denotes the Bernstein matrix associated to the $i$-th row of the the polynomial matrix $\mathcal{H}$ after mapping it to the unit box $U$ (with respect to the degree $\delta\in \N^n$ equal to the maximal degrees of $\mathcal H$). Consider the set $\hat{O}$ defined as the feasible values of
$\theta$ that satisfy the following constraints
\begin{equation}
\label{eq:inputcons1}
  \alpha_{\mathcal{U},k}\cdot \mathcal{H}_{i}\theta \le \beta_{\mathcal{U},k},\ k\in \mathcal K_{\mathcal U},\; \forall i=1,\dots,m \,.
 \end{equation}

\begin{lemma}
$\hat{O} \subseteq O$.
\end{lemma}

Now we will show that finding the feasible sets $C'$ and $O'$ leads to a bilinear program.
First,  we can find a polynomial matrix $B(x)$ to allow us to write
$$-\grad V_c(x) \cdot (f(x)+G(x)\theta)=c^t\ {B(x)} \tilde{\theta} ,$$
where $ \tilde{\theta}=\left(\begin{array}{c}
 1 \\ \theta
\end{array}\right)$ and $B(x)=(\grad V_m(x))^t\cdot (f(x) \quad G(x))$. Here  $(\grad V_m(x))$ denotes the matrix where each column corresponds
to the Jacobian of one of the monomials of the Lyapunov function.

The main difference with the previous case is that instead of the vector of monomials $m$ we have ${B(x)} \tilde{\theta}$. The degree $\delta$
will be chosen as the maximal degrees of the polynomials in $B(x)$.
By consequence, the Bernstein conversion matrix will be a set of $n$ matrices $\scr{B}_{\theta,i}= {\scr{B}_i}\tilde{\theta}$ 
where  $\scr{B}_i$ is the Bernstein conversion matrix corresponding to the polynomial row $B_i(x)$ of the polynomial matrix $B(x)$ after mapping
it to the unit box $U$.
%\begin{itemize}
%For each  we compute .
%\item $\scr{B'}$ is computed by composing the $\scr{B}_i$ matrices.
%\end{itemize}
Now using the same ideas as previously we will get by applying Farkas lemma a set of linear programs:
\begin{lemma}\label{lem:feasibility1}
$c$ is a solution to the problem in~\cref{eq:rlt-parameterized} if and only if there exist multipliers $c$ and $\lambda^{i}$
such that 
\begin{equation}
\label{eq:feasability2}
  A^t {\lambda}^{i} = - {\scr{B}_{\theta,i}}^t c,\; 
  b^t \lambda^{i} \leq 0,\; \mbox{ and } \lambda^{i} \geq 0, \mbox{ for all } i=1,\dots,n.
 \end{equation}
\end{lemma}

Since $\scr{B}_{\theta,i}= \scr{B}_i'\tilde{\theta}$, the previous
lemma give us a set of \emph{bilinear feasibility problems} for the
feasible sets $C'$ and $O'$. But checking feasibility and solving a
bilinear program is well-known to be NP-hard~\cite{GareyJohnson}. Rather
than solve these problems directly, we consider a policy iteration approach
in Section~\ref{Sec:poly-lyap-synth}.

\subsection{Linear feasibility programs for positive invariance (Problem~\ref{pb2})}
We now turn to the problem of encoding the invariance of the region $R$.
Our approach reuses ideas from earlier work by Ben Sassi and Girard
using the blossoming principle to enforce the invariance of a polytope
for a polynomial system~\cite{Bensassi3}. We obtain linear constraints 
over $\theta$ that define a feasible region $\hat{O}_F \subseteq O_F$
such that choosing any $\theta \in \hat{O}_F$ guarantees that the region $R$ 
will be maintained invariant.

First, will need to define a facet and its outer normal~\cite{Belta06} for a general rectangle $R_n=\prod_{k=1}^{k=n}[a_k,b_k]$:
\begin{itemize}
\item $\xi_k:\{a_k,b_k\}\mapsto\{0,1\} $ when for all $k\in\{1,\ldots,n\}$, $\xi_k(a_k)=0$ and $\xi_k(b_k)=1$. 
\item$F_{j,\xi_j(w_j)}=\{x\in R_n\; | \;x_j=w_j\}$: the set of facets of $R_n$
 where for all $j\in\{1,\ldots,n\}$, $w_j\in\{a_j,b_j\}$. 
\item$n_{j,\xi_j(w_j)}=(-1)^{(\xi_j(w_j)+1)}e_j$: the outer normal of the facet $F_{j,\xi_j(w_j)}$ where 
the vectors $e_j$ form the canonical basis of $R^n$.
\end{itemize}
For the invariance context, all the results are derived
from~\cite{Bensassi3} so they are given without demonstration. We
simply adapt the main result (Theorem 6 in~\cite{Bensassi3}) to the
specific form of the controller required in this work.  For doing so
we define for a fixed degree $\delta=(\delta_1,\dots,\delta_n)$, for
all $j\in\{1,\dots,n\}$ and all $l\in\{1,\dots,\delta_j\}$ :
 $$I_{j,l}=\{ I=(i_1,\dots,i_n)\in \N^n, \text{ such that }I \le \delta \text{ and } i_j=l\}.$$
More precisely, we need to replace in~\cite{Bensassi3} the vector field $f$ by $f+G\theta$ and the blossom values by the Bernstein coefficients.
Let $f_U$ and $G_U$ denote the polynomial vector field $f$ and the polynomial matrix $G$ after mapping 
them to the unit box $U$ and let $f_{U,I}$ and $G_{U,I}$ the associated Bernstein coefficient vector and matrix for all multi-indice $I\le \delta$. 
We will obtain the following result:
\begin{corollary}
\label{coro:inv}
For all $j\in\{1,\dots,n\}$, we have:
\begin{enumerate}
\item The facet $F_{j,\xi_j(a_j)}$ of the rectangle $R_n$ is blocked for the controlled system $\dot{x} = f+G\theta$ 
 if $f_{U,I,j}+G_{U,I,j}\theta\ge 0 $ for all $I \in I_{j,0}$.
\item The facet $F_{j,\xi_j(b_j)}$ of the rectangle $R_n$ is blocked for the controlled system $\dot{x} = f+G\theta$ 
if $f_{U,I,j}+G_{U,I,j}\theta\le 0 $ for all $I \in I_{j,\delta_j}$,
\end{enumerate}
where  $f_{U,I,j}$ and $G_{U,I,j}$ are respectively the $j$ component (row) of the vector $f_{U,I}$ (matrix $G_{U,I}$). 
 \end{corollary}
 The corollary gives us a linear program allowing to compute the feasible sets $\hat{O}_F$ for all facets $F\in \mathcal F$.

\section{Joint synthesis of polynomial Lyapunov functions and controllers}\label{Sec:poly-lyap-synth}
First of all, we are going to present an algorithm to solve our stability problem, then we will show how the results can be improved
by using a decomposition criterion and extend the results using this decomposition to a particular class of hybrid systems.
\subsection{Algorithm}
In this section, we give an algorithm allowing to summarize the previous results in order to solve our stabilization problems by
synthesizing jointly the controller that stabilize the system and the Lyapunov function for the controlled system.\\
In fact, the main problem when regrouping the feasibility problems of the previous section is that we have to deal with a bilinear program for which
there is no practical way to solve it. We will define an iterative approach where for each step one of the parameters ($\theta$ for the controller or $c$ for the Lyapunov function) is fixed and the other is computed by solving a linear program. 
The overall approach is given as follows:
\begin{enumerate}
\item Initialize $\theta^*=0$. 
\item Compute feasible set $C$ using feasibility problem (\ref{eq:feasability1}). 
\item Find a "maximal" coefficient vector $c\in C$ for the Lyapunov function: \\
We fix $\theta=\theta^*$ and we solve the  feasibility problems (\ref{eq:feasability2}) by relaxing $"\le 0"$ by $"\le t"$
where $t$ will be a positive decision variable to be minimized. The outputs of the linear program are $(c^*, t^*)$.
%If $t^* \approx 0$ then $V_{c^*}$ certify the global stability of the controlled system. STOP if $\theta^*\neq 0$.
\item Find a "maximal" coefficient vector $\theta$ for the controller: \\
We fix $c=c^*$ and we solve the feasibility problems given by the (RHS) of (\ref{eq:inputcons1}) and the ones of Corollary~\ref{coro:inv}.
By using the same idea of relaxing $"\le 0"$ by $"\le t"$ for a positive decision variable $t$ and minimize over $t$, we get outputs $({\theta}^*, t^*)$.
If $t^* \approx 0$ STOP , else Go back to the previous step.
\end{enumerate}
When the algorithm terminates, the outputs $(c^*,\theta^*)$ will give us the admissible controller and the Lyapunov function proving the asymptotic stability
of the controlled system. The invariance problem of the rectangular domain will be ensured.
\subsection{Decomposition and generalization for a particular class of hybrid systems}
As mentioned in~\cite{Bensassi+Sriram}, the Bernstein relaxation (\ref{eq:lpbern2}) can be much more efficient once a good decomposition 
is provided. By "good" we mean a box decomposition where local minima will belong to the edge of the box. Since the global minimum of the 
Lyapunov function is known in advance ($0_n$ in our case), a decomposition of the rectangle $R$ around zero (by putting zeros on the edges of the resulting rectangles)
will significantly improve the precision of the approach. The drawback is that $2^n$ decomposition are needed. In fact by using this decomposition, each feasibility
problem in the previous algorithm (except the invariance ones) will be replaced by $2^n$ feasibility problems.\\
Now, since the approach deals with a box partition of the state space, one can easily extend the dynamical system (\ref{eq:ode}) to 
the following class of hybrid system where the state space is decomposed to boxes and each box has its own polynomial dynamic. More precisely, for all $i\le 2^n$, let $R_i$ 
be the set of boxes of our 'zero' decomposition and the hybrid system will be following :
\begin{equation}
\label{eq:odeg}
\left\{
\begin{array}{ll}
\dot x_i(t) = f_i(x(t))+G(x) \theta_i,\; \theta_i\in  O \;  x\in R_i.\\
\end{array}\right.
\end{equation}
The difference here is that each of the $2^n$ feasibility problems (Step 4) will provide an admissible controller $\theta_i$ trying to make 
the Lyapunov function decreasing in the corresponding box. So we will get a common Lyapunov function having multiple derivatives (one
for each box). Also we should remark that when dealing with the invariance problem, linear feasibility problems of Corollary~\ref{coro:inv} should be adapted. In fact, for each 
box one should ensure the feasibility problems with respect to the facets that should be blocked.
\begin{remark}
The previous result will hold for each other box decomposition. In fact we can always be reduced to the previous case by decomposing each sub
box containing $0_n$ into sub boxes where $0_n$ will belong to the edges.
\end{remark}
%%\subsection{}

\section{Numerical results}\label{Sec:numerical-eval}
\subsection{Illustrative example}
To illustrate the approach, we consider the following $2$-dimensional polynomial system and a box $R=[-1,1]^2$. 
\begin{equation*}
\left\{
\begin{array}{rcll}
 \dot{x}_1 &=&f_1(x)= x_2-x_1^2 +3x_2^2-2x_1x_2, \vspace{2mm}\\
 \dot{x}_2 &=&f_2(x)=-x_1-3x_1^2+x_2^2+2x_1x_2 .
\end{array}\right.
\end{equation*}
By simulation, one can see that the origin is not asymptotically stable and that the box $[-1,1]^2$ is not invariant for the system (see Figure~\ref{fig1}). 
\begin{figure}[!h]
\begin{center}
\includegraphics[angle=0,width=0.4\textwidth]{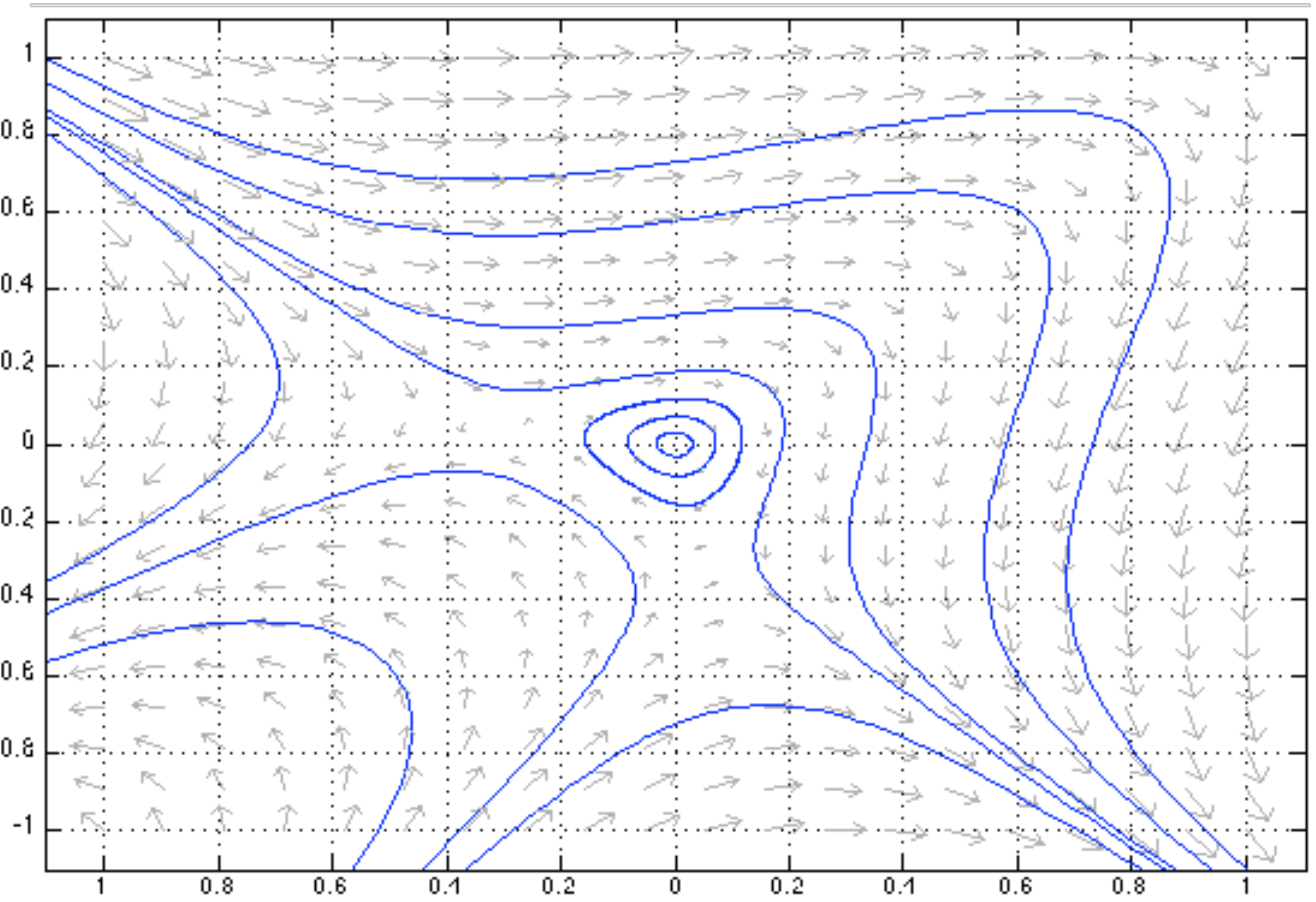} 
\caption{Vector fields and some trajectories of the uncontrolled system.} %la l�gende
\label{fig1} 
\end{center}
\end{figure} 
Using our approach, we aim to find a linear state feedback controller ensuring the asymptotic stability of the origin
and the invariance of $R$. We will consider the following controlled system:
\begin{equation*}
\dot x(t) = f(x(t))+g(x(t)) u(x(t)),\; 
\end{equation*}
where $g(x)=I_2=\left( {\begin{array}{cc}
   1 & 0 \\
    0 & 1  \\
  \end{array} } \right)$ and $u(x)=Ax$ where $A= \left( {\begin{array}{cc}
   a_{11} & a_{12} \\
    a_{21} & a_{22}  \\
  \end{array} } \right)$.\\
 Since we look for a linear state feedback controller, we can write $u(x)=H(x)\theta$ where \\
  $H(x)= \left( {\begin{array}{cccc}
   x_1 & x_2 & 0 & 0\\
    0 & 0 & x_1 &x_2 \\
 \end{array} } \right)$ and $\theta=(a_{11}, a_{12}, a_{21}, a_{22})^\top$.\\
For the Lyapunov function, we fix the following form :
$$V_c(x)=c_1x_1+c_2x_2+c_3x_1^2+c_4x_2^2+c_5x_1x_2+c_6x_1^4+c_7x_2^4.$$
We impose that $-5 \le c_i \le 5$ for all $i \in\{1,\dots,7\}$ and add the fact that
$ c_i \ge 0.01$  for all $i \in\{3,4\}$ in order to ensure that $V$ is positive definite.
Also the linear coefficients of the controller are bounded by $-5$ and $5$.\\
The iterative approach needs two iterations to globally stabilize $R$. Outputs are :
\begin{itemize}
\item $A= \left( {\begin{array}{cc}
   -4.5471& 0.7000  \\
     3.9290 & -4.6218  \\
  \end{array} } \right)$.
\item $V(x)=0.01(x_1^2+x_2^2)+0.009x_1x_2+0.036x_1^4+0.023x_2^4.$
\end{itemize}
One can simulate the obtained system and verify the asymptotic stability and the invariance of $R$ (see Figure~\ref{fig1}).
\begin{figure}[!h]
\begin{center}
\includegraphics[angle=0,width=0.4\textwidth]{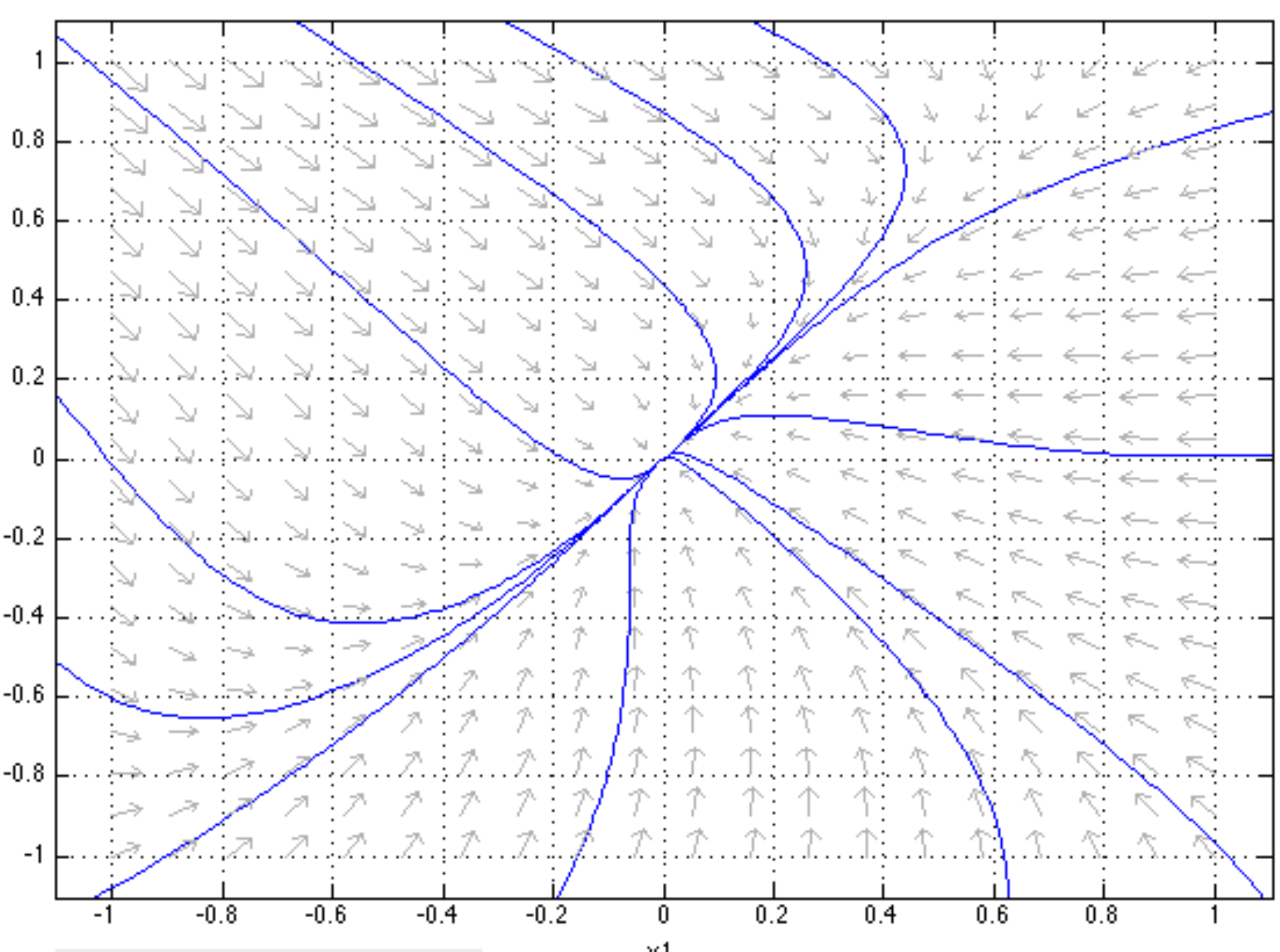} 
\caption{Vector fields and some trajectories of the controlled system.} %la l�gende
\label{fig1} 
\end{center}
\end{figure} 
Now, we will use the approach to deal with the Hybrid case. More precisely, we decompose $R$ around zero ($R_1=[-1,0]^2$, 
$R_2= [-1,0]\times[0,1]$, $R_3= [0,1]\times[-1,0]$, $R_4= [0,1]^2$) and try to find for each sub-box $R_i$ a linear controller $u_i$ such that 
the following Hybrid system
\begin{equation*}
\dot x(t) = f(x(t))+g(x(t)) u_i(x),\; 
\end{equation*}
is globally stable with respect to $R$ where $u_i(x)=A_ix$ for all $x \in R_i$ and all $i\in\{1,\dots,4\}$ .\\
In this case, only one iteration is needed to stabilize the system inside $R$ since we have more freedom in the choice of the controller. Outputs are :
\begin{itemize}
\item $A_1= \left( {\begin{array}{cc}
      -4.4721 & -3.4219\\   -2.9376  & -4.0957 \\
    \end{array} } \right)$.
\item  $A_2= \left( {\begin{array}{cc}  
  -4.3795&   0.3130\\    1.1904 &  -4.3770 \\  
  \end{array} } \right)$.
\item  $A_3= \left( {\begin{array}{cc}  
  -4.3331  &  2.6016 \\   3.3924 &  -4.2926 \\
   \end{array} } \right)$.
\item     $A_4= \left( {\begin{array}{cc}  
  -4.1427  & -3.0418\\   -3.3052  & -4.4195 \\
   \end{array} } \right)$.
 \item $V(x)= 4.7737x_1^2+4.7743x_2^2+4.8172x_1^4+4.8175x_2^4.$
 \end{itemize}
By simulating trajectories in those boxes, we can verify that the stability property and the box invariance hold (see Figure~\ref{fig2} for $R_1$ and $R_2$).
\begin{figure}[!h]
\begin{center}
\includegraphics[angle=0,width=0.4\textwidth]{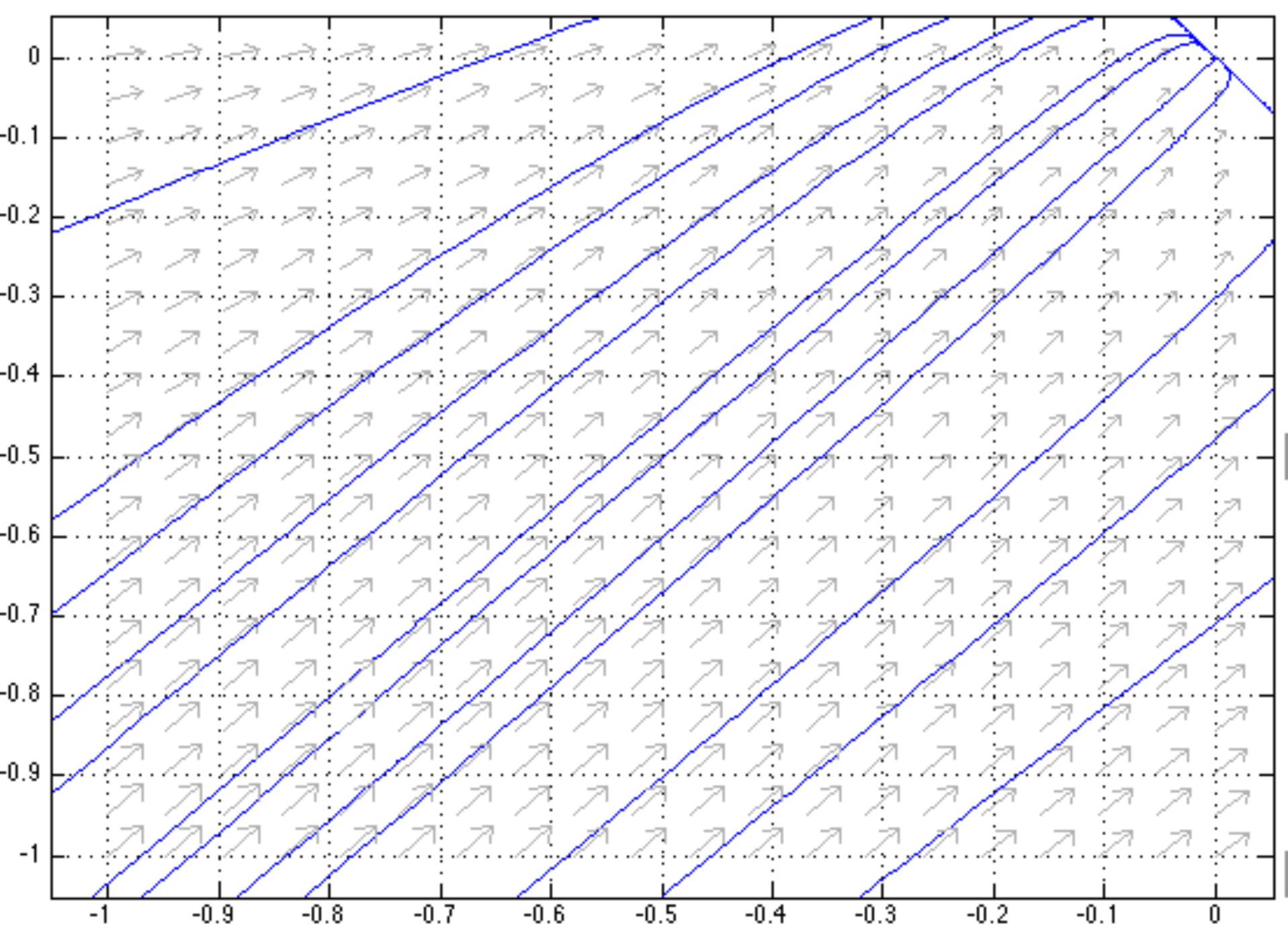} \includegraphics[angle=0,width=0.4\textwidth]{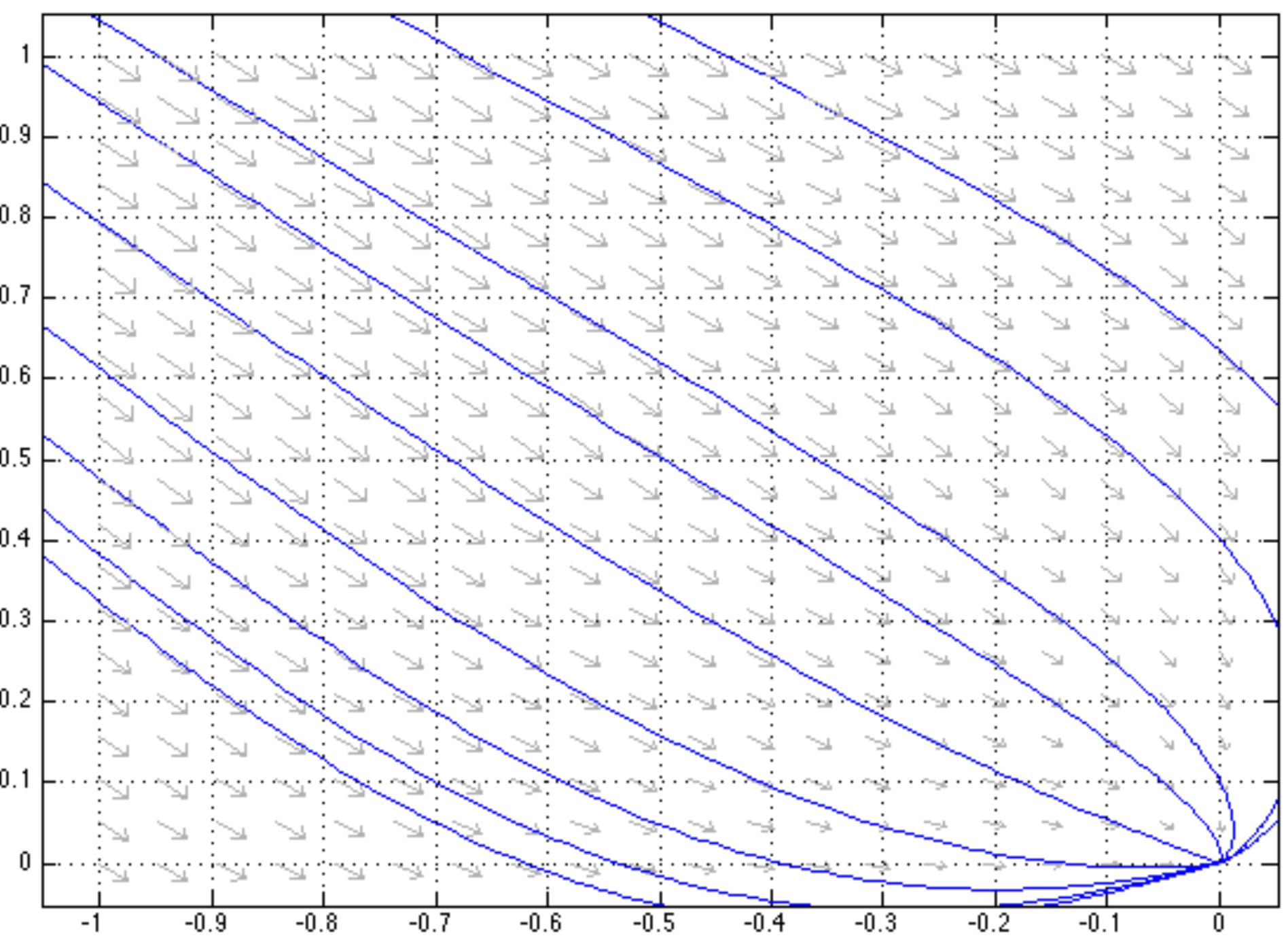} 
\caption{Vector fields and some trajectories of the controlled system associated to $R_1$ (on the top) and $R_2$ (on the bottom)} %la l�gende
\label{fig2} 
\end{center}
\end{figure} 
\subsection{Benchmarks}
We discuss the results obtained for a set of benchmarks borrowed from the literature. We run the algorithm until a good precision $\epsilon$ \footnote{$\epsilon$ 
denotes the precision $t^*$ of the algorithm.}is reached or a fixed number of iterations (the approach fails to make progress). In the latter case one can add more flexibility in the
templates by adding terms of higher degrees. In failure cases, we remove the invariance constraints in order to achieve just the asymptotic stability property. We report separately
stability (Stab column) and invariance (Inv column).  A threshold of precision around $10^{-6}$
is considered to confirm that the property holds. We report also the number of iteration needed to achieve the given precision.
A detailed description of the systems, explicit expression of Lyapunov functions and controllers are given in the Appendix.

\begin{table}[!h]
\caption{Table showing performance of our method on a set of benchmarks. }\label{Tab:table-ex-1}
\begin{center}
\begin{tabular}{|r|l|l|l|l||l|l|l|}

\hline
 Id & $R$ & $\mathcal U$  & $\epsilon$ & Stab &Inv & Iter\\
\hline
 1&    $[-0.5,0.5]^2$    &$[-1,1]$       &$4*10^{-19}$ &\tick    &\cross   & 1   \\
 2&    $[-1,1]^2$          &$[-2,2]$       &$4*10^{-9}$   &\tick    &\tick      & 2   \\
 3&    $[-1,1]^2$          &$[-4,4]$       &$2*10^{-17}$ &\tick    &\cross   & 1   \\
 4&    $[-1,1]^2$          &$[-1,1]$       &$4*10^{-7}$   &\tick    &\tick      & 3     \\
 5&    $[-1,1]^3$          &$[-10,10]$   &$2*10^{-6}$   &\tick    &\tick      & 4   \\
 6&    $[-0.5,0.5]^3$    &$[-5,5]$       &$2*10^{-7}$   &\tick    &\cross   &6     \\ 
 7&    $[-0.5,0.5]^3$    &$[-3,3]$       &$9*10^{-7}$   &\tick    &\cross   &3     \\
 8&    $[-0.5,0.5]^3$    &$[-1,1]$       &$4*10^{-5}$   &?    &\cross       &9    \\
 9&    $[-0.1,0.1]^4$    &$[-5,5]$       &$5*10^{-5}$   &?    &\cross       &3   \\
 10&   $[-0.1,0.1]^4$&$[-10,10]$       &$8*10^{-5}$  &?    &\cross       &4   \\
11&   $[-0.05,0.05]^5$&$[-1,1]$       &$6*10^{-5}$   &?    &\cross       & 2   \\

\hline
\end{tabular}
\end{center}
\end{table}
Note that that invariance conditions usually make the feasibility of the approach very restricted since it needs to holds 
simultaneously with the stability conditions. This explains the fact that only few stabilazable systems can only have the invariance box 
property. The computation time is roughly in size of the problem and the templates: roughly each iteration of two dimensional systems (systems $1,2,3,4$) required almost one second, for three dimensional systems it required between two and three seconds (systems $4,5,6,7$).
\section{Appendix}

\begin{example}(see~\cite{Liberzon1999})
\begin{equation*}
\left\{
\begin{array}{ll}
\dot{ x}= y. \\
\dot{y} = -x + u(y).
\end{array}
\right.
\end{equation*}
\begin{itemize}
\item $u(y)=-2y$.
\item $V(x,y)=0.01(x^2+y^2)$ 
\end{itemize}
\end{example}
\begin{example}(see Lectures on back-stepping\footnote{\url http://control.ee.ethz.ch/~apnoco/Lectures2014})
\begin{equation*}
\left\{
\begin{array}{ll}
\dot{x} = y - x^3.\\
\dot{y} = u(x,y).
\end{array}
\right.
\end{equation*}
\begin{itemize}
\item $u(x,y)=-x-\frac{2}{3}y+\frac{1}{3}x^3$ .
\item $V(x,y)=0.01(y^2+x^2y^2)+0.0102x^2+0.0007 xy$. 
\end{itemize}
\end{example}
\begin{example}
\begin{equation*}
\left\{
\begin{array}{ll}
\dot{x} = y \\
\dot{y} = u(y)y^2-x.
\end{array}
\right.
\end{equation*}
\begin{itemize}
\item $u(y)=4(y^2-y)$.
\item $V(x,y)=0.01(x^2+y^2+x^2y^2)+0.005(x^4+y^4)$.
\end{itemize}
\end{example}
\begin{example}(See~\cite{Perruquetti1995})
\begin{equation*}
\left\{
\begin{array}{ll}
\dot{x} = -x(0.1+(x+y)^2) \\
\dot{y} =(u(x)+x)(0.1+(x+y)^2).
\end{array}
\right.
\end{equation*}
\begin{itemize}
\item $u(x)=-x$.
\item $V(x,y)=0.01(y^2+x^2y^2)+0.0657x^2+0.0022xy+0.0019y^4$.
\end{itemize}
\end{example}
\begin{example}\begin{equation*}
\left\{
\begin{array}{ll}
\dot{x} = y+0.5z^2. \\
\dot{y} =z.\\
\dot{z} =u(x,y,z).
\end{array}
\right.
\end{equation*}
\begin{itemize}
\item $u(x,y,z)=-0.59185x-5.9217y-0.51825z+0.061785x^2+0.12415xy-0.4642xz+0.048453x^3-0.57345y^3$.
\item $V(x,y,z)=0.01x^2+0.0583y^2+0.0099z^2+0.0134xy+0.003xz+0.004y^4+0.0024yz+0.0003z^4$.
\end{itemize}
\end{example}
\begin{example}(See~\cite{Yeom2012})
\begin{equation*}
\left\{
\begin{array}{ll}
\dot{x} = -x+y-z. \\
\dot{y} =-x(z+1)-y.\\
\dot{z} =-x+u(x,z).
\end{array}
\right.
\end{equation*}
\begin{itemize}
\item $u(x,z)=1.76524x-4.7037z$.
\item $V(x,y,z)=0.01(x^2+y^2)+0.013z^2$.
\end{itemize}
\end{example}
\begin{example}(see Lectures on back-stepping)
\begin{equation*}
\left\{
\begin{array}{ll}
\dot{x} = -x^3+y. \\
\dot{y} =y^3+z.\\
\dot{z} =u(x,y,z).
\end{array}
\right.
\end{equation*}
\begin{itemize}
\item $u(x,y,z)=-0.083339x-3.5413y-0.33868z-0.4325x^3$.
\item $V(x,y,z)=0.01(x^2+z^2)+0.0333z^2+0.0033xy+0.0048xz+0.0061yz$.
\end{itemize}
\end{example}

\begin{example}(See~\cite{Yeom2012})
\begin{equation*}
\left\{
\begin{array}{ll}
\dot{x} = z^3-y. \\
\dot{y} =z.\\
\dot{z} =u(x,y,z).
\end{array}
\right.
\end{equation*}
\begin{itemize}
\item $u(x,y,z)=-0.86597x-0.16208y-0.61597z$.
\item $V(x,y,z)=0.01(x^2+z^2)+0.0333z^2+0.0179xy+0.0129xz+0.0127y^2$.
\end{itemize}
\end{example}

\begin{example}
\begin{equation*}
\left\{
\begin{array}{ll}
\dot{x} = y. \\
\dot{y} =-0.1y-10z+xv^2.\\
\dot{z} =v.\\
\dot{v} =-z-v+u(x,y,z,v).\\
\end{array}
\right.
\end{equation*}
\begin{itemize}
\item $u(x,y,z,v)=-12.0271x-8.1243y-10.2755z-10.047v$.
\item $V(x,y,z,v)=0.1202x^2+0.01(y^2+v^2)+0.2201z^2+0.2556xz+0.0101xv+0.01578yz+0.0115yv$.
 \end{itemize}
\end{example}

\begin{example}(Ball and Beam example~\cite{Raja})
\begin{equation*}
\left\{
\begin{array}{ll}
\dot{x} = y. \\
\dot{y} =-9.8z+1.6z^3+xv^2.\\
\dot{z} =v.\\
\dot{v} =u(x).\\
\end{array}
\right.
\end{equation*}
\begin{itemize}
\item $u(x)=-6x$.
\item $V(x,y,z,v)=0.0672x^2+0.01y^2+0.1074z^2+0.0136v^2-0.0043xy+0.149xz+0.0023xv+0.008yz+0.0189yv-0.003zv$.
 \end{itemize}
\end{example}
\begin{example}
\begin{equation*}
\left\{
\begin{array}{ll}
\dot{x} = -0.1x^2-0.4xv-x+y+3z+0.5v. \\
\dot{y} =y^2-0.5yw+x+z.\\
\dot{z} =0.5z^2+x-y+2z+0.1v-0.5w.\\
\dot{v} =y+2z+0.1v-0.2w+u(x,y,z,v,w).\\
\dot{w} =z-0.1v+u(x,y,z,v,w).
\end{array}
\right.
\end{equation*}
\begin{itemize}
\item $u(x,y,z,v,w)=-1.5x-1.5y-1.5z-1.5v-1.5w.$ 
\item $V(x,y,z,w,v)=0.01(x^2+y^2+v^2+w^2)-0.0066xy-0.0252xz-0.008(xv+yv)+0.005xw+0.001yz+0.0167yw-0.0023zv-0.0121zw+0.001vw$.
\end{itemize} 
\end{example}

\section{Conclusion}
In this paper a linear programming approach is presented allowing to deal with the stabilization problem of polynomial systems.
The approach is based on Bernstein polynomials and propose a policy iteration technique allowing to avoid bilinear programs by
having an iterative approach of linear programs instead. The benchmarks results show that the method can be efficient in  practice. 
The drawback of this technique is that no convergence result is guaranteed and even in case of convergence there is no guaranty that it will be to a local minima. 
A future work will be a deeper study of the failure case or the fix point (once the algorithm result does not improve): an idea is to fix small variation for each variable of the
bilinear program and try to find a descent direction helping the algorithm to improve.
%To conclude, we have examined three different LP relaxations for
%synthesizing polynomial Lyapunov functions for polynomial systems. We
%compare these approaches to the standard approaches using
%Schm{\"u}dgen and Putinar representations that are used in SOS
%programming relaxations of the problem. In theory, the Schm{\"u}dgen
%representation approach subsumes the three LP relaxations. In
%practice, however, we are forced to use the Putinar representation. We
%show that the Putinar representation can prove some polynomials
%positive semi-definite that our approaches fail to. On the other hand,
%the reverse is also true: we demonstrate a polynomial that is easily
%shown to be positive semi-definite on the interval $[-1,1]^n$ through
%LP relaxations. However, the same fact cannot be demonstrated by a
%Putinar representation approach. We then compare both approaches over
%a set of numerical benchmarks. We find that the LP relaxations succeed
%in finding Lyapunov functions for all cases, while the Putinar
%representation fails in many benchmarks due to numerical (conditioning)
%issues while solving the SDP. As future work, we wish to extend our
%approach to a larger class of Lyapunov functions.  We also are looking
%into the problem of analyzing systems with non-polynomial dynamics and
%the synthesis of non-polynomial Lyapunov functions.

%%\bibliographystyle{plain}
%\bibliographystyle{plainnat}
%\bibliography{biblyap}
%%%\newpage
%\appendix
%\section{Description of Synthesized Benchmarks}\label{App:description}
%\input{appendix1.tex}

\end{document}